\documentclass[a4paper,11pt]{article}
\usepackage{geometry} 

\usepackage{verbatim}
\usepackage{fancyhdr}       
\usepackage[colorlinks, citecolor=blue]{hyperref}           
\usepackage {graphics}
\usepackage{subfigure}
\usepackage{float}
\usepackage{graphicx}
\usepackage{mathrsfs}
\usepackage{indentfirst, latexsym, bm, color}
\usepackage{amsmath, amssymb, amsfonts,amsthm}
\usepackage{caption}
\usepackage[all]{xy}
\usepackage{setspace}
\usepackage[named]{algorithm}
\usepackage{authblk}

\makeatletter
\newcommand*\bigcdot{\mathpalette\bigcdot@{.7}}
\newcommand*\bigcdot@[2]{\mathbin{\vcenter{\hbox{\scalebox{#2}{$\m@th#1\bullet$}}}}}
\makeatother

\theoremstyle{plain}
\theoremstyle{definition}
\newtheorem{definition}{Definition}
\newtheorem{lemma}[definition]{Lemma}
\newtheorem{theorem}[definition]{Theorem}
\newtheorem{proposition}[definition]{Proposition}
\newtheorem{corollary}[definition]{Corollary}
\newtheorem{example}[definition]{Example}
\newtheorem{remark}[definition]{Remark}

\newtheorem*{theoremA}{Theorem A}
\newtheorem*{theoremB}{Theorem B}
\newtheorem*{theoremC}{Theorem C}
\newtheorem*{them}{Theorem}

\title{Dolbeault-type complexes on $G_2$- and $\mathrm{Spin}(7)$-manifolds}
\author{Xue Zhang \thanks{zxue@tsinghua.edu.cn}}
\affil{\footnotesize\emph{Department of Mathematical Sciences, Tsinghua University, Beijing, China}}

\begin{document}

\maketitle

\begin{abstract}
There are three types of Dolbeault complexes arising from representations of holonomy group on a Riemannian manifold, two of which are dual to each other. Such a complex is elliptic if and only if its generator satisfies an algebraic condition. We list all Dolbeault complexes on compact $G_2$- and $\mathrm{Spin}(7)$-manifolds. Each cohomology group can be described by harmonic forms.
\end{abstract}

\section{Introduction}
It is well known that a de Rham complex on a smooth manifold $M$ is the cochain complex $(A^*(M),\mathrm{d})$ of differential forms on $M$ with the exterior derivative as its differential. If $M$ carries a Riemannian metric with special holonomy, each $A^k(M)$ possesses a direct sum decomposition. For instance, $A^k\otimes\mathbb{C}$ can split into $\oplus_{p+q=k}A^{p,q}$ with respect to a $U(n)$-action and there are Dolbeault complexes on complex manifolds. When the holonomy is a symplectic group, each quaternionic manifold owns a Dolbeault-type double complex  \cite{Widdows2002}. In Berger's classification \cite{Berger1955} of holonomy groups of a nonsymmetric, irreducible Riemannian manifold, there are two exceptional holonomy groups, $G_2$ in $7$ dimensions and $\mathrm{Spin}(7)$ in $8$ dimensions. An interesting question is that, how can one find all de Rham subcomplexes on a $G_2$ or $\mathrm{Spin}(7)$-manifold? More precisely,  
search for a sequence of subbundles $E_i\subset \Lambda^i(M),i=0,1,2,\cdots,$ such that the spaces of smooth sections of $E_i$ composes an elliptic complex (ref. \cite{AtiyahBott1967})
\begin{eqnarray}\label{firstellipc5}
0\longrightarrow\Gamma(E_0)\stackrel{D_0}{\longrightarrow}\Gamma(E_1)\stackrel{D_1}{\longrightarrow}\Gamma(E_2)\stackrel{D_2}{\longrightarrow}\cdots
\end{eqnarray}
with boundary operator $D_i=p_i\circ\mathrm{d},$ where each $p_i
:A^i(M)\rightarrow\Gamma(E_i)$ is an orthogonal projection with respect to a Riemannian metric. In general, this kind of complexes (\ref{firstellipc5}) can be divided into three types due to the condition $D_{i+1}D_i=0,$
\begin{itemize}
\item type \uppercase\expandafter{\romannumeral1}
\begin{eqnarray}\label{firellty1}
0\rightarrow A^0\stackrel{\mathrm{d}}{\rightarrow}A^1\stackrel{\mathrm{d}}{\rightarrow}\cdots\stackrel{\mathrm{d}}{\rightarrow}A^{k-1}\stackrel{D_0}{\rightarrow}\Gamma(E_0^\perp)\stackrel{D_1}{\rightarrow}\Gamma(E_1^\perp)\stackrel{D_2}{\rightarrow}\Gamma(E_2^\perp)\stackrel{D_3}{\rightarrow}\cdots,
\end{eqnarray}
\item type \uppercase\expandafter{\romannumeral2} 
\begin{eqnarray}\label{firellty2}
0\rightarrow \Gamma(F)\stackrel{D_0}{\rightarrow}\Gamma(E_0^\perp)\stackrel{D_1}{\rightarrow}\Gamma(E_1^\perp)\stackrel{D_2}{\rightarrow}\Gamma(E_2^\perp)\stackrel{D_3}{\rightarrow}\cdots,
\end{eqnarray}
\item type \uppercase\expandafter{\romannumeral3} 
\begin{eqnarray}\label{firsellty3}
0\rightarrow\Gamma(E_0)\stackrel{D_0}{\rightarrow}\Gamma(E_1)\stackrel{D_1}{\rightarrow}\Gamma(E_2)\stackrel{D_2}{\rightarrow}\cdots\stackrel{\mathrm{d}}{\rightarrow}A^{n-1}\stackrel{\mathrm{d}}{\rightarrow}A^{n}\rightarrow 0,
\end{eqnarray}
\end{itemize}
where $E_i=E_0\cdot\Lambda^i\subset\Lambda^{k+i}$ for $i=0,1,2,\cdots$ and $F\subset \Lambda^{k-1}.$ Type \uppercase\expandafter{\romannumeral1} and type \uppercase\expandafter{\romannumeral3} are dual to each other, that is, (\ref{firsellty3}) is the quotient complex of $A^*(M)$ by (\ref{firellty1}). We say that (\ref{firellty1}) and (\ref{firellty2}) are generated by $E_0$ and $(F,E_0)$ respectively. We write the complex by its generator if there is no ambiguity.

Salamon \cite{Salamon1989} and Reyes-Carri\'{o}n \cite{Reyes1998} generalized the notion of instanton to high dimensional Riemannian manifolds with special holonomy groups. Reyes-Carri\'{o}n showed in \cite{Reyes1998} that:
\begin{them}[Reyes-Carri\'{o}n]
Assume that $G\subset\mathrm{SO}(n)$ is a special holonomy group and the holonomy group of a $n$-manifold $(M,g)$ reduces to the normalizer of $G,$ then the complex generated by $\Lambda(\mathfrak{g})$ is elliptic. Here $\Lambda(\mathfrak{g})$ is a subbundle of $\Lambda^2(M)$ associated with the adjoint representation of $G.$
\end{them}
A key point in Reyes-Carri\'{o}n's result is that the fibers of generators of complexes are invariant under the action of holonomy group. Based on it, we are concerned in this paper with the action of holonomy group on exterior algebra. We introduce the notion of locally quasi-flat (in \S \ref{pdrcdef5}) and complete prime (in \S \ref{firssectionjy}). The condition $D^2=0$ for $(E,D)$ and $(F,E,D)$ is equivalent to that $E$ is locally quasi-flat. The elliptic conditions for $(E,D)$ and $(F,E,D)$ are equivalent to that $E$ and the pair $(F,E)$ are complete prime. We state our first result:
\begin{theoremA}
Let $G$ be a connected subgroup of $\mathrm{SO}(n)$ and let $(M,g,G,Q)$ be a closed Riemannian $n$-manifold with a compatible torsion-free $G$-structure $Q.$ Suppose $\mathcal{F}\subset\Lambda^{k-1}(\mathbb{R}^n)$ and $ \mathcal{E}\subset\Lambda^k(\mathbb{R}^n)$ are invariant subspaces under the action of $G.$ Set $F=Q\times_G\mathcal{F}\subset\Lambda^{k-1}(M)$ and $E=Q\times_G\mathcal{E}\subset\Lambda^k(M).$ Then $E$ is locally quasi-flat. Consequently,
\begin{itemize}
\item[(\romannumeral1)] If $\mathcal{E}$ is complete prime, then $(E,D)$ is an elliptic complex. 
\item[(\romannumeral2)] If $(\mathcal{F},\mathcal{E})$ is complete prime, then $(F,E,D)$ is an elliptic complex. 
 \end{itemize}
\end{theoremA}
Suppose $(M,g,G,Q)$ is as in Theorem A. The action of $G$ on $\Lambda^*(\mathbb{R}^n)$ splits $\Lambda^*(\mathbb{R}^n)$ into $\oplus_{k=0}^n\oplus_{l\in I_k}\Lambda_l^k$ an orthogonal direct sum of irreducible representations of $G,$ where each $\Lambda_l^k$ is of dimension $l.$ Also, we use the notation $\Lambda_l^k(M)$ for the subbundle $Q\times_G\Lambda_l^k$ of $\Lambda^k(M)$ and use $A_l^k$ for the space of smooth sections of $\Lambda_l^k(M).$ When $M$ is a $G_2$-manifold or $\mathrm{Spin}(7)$-manifold, the decomposition of $\Lambda^*(M)$ can be given explicitly, see \cite{Salamon1989},\cite{Joyce1996} or \cite{DJoyce1996}.  Apply Theorem A to $G_2$ and $\mathrm{Spin}(7)$-manifolds, we have
 \begin{theoremB}
Let $(M,g,\varphi)$ be a compact $G_2$-manifold. Then the following complexes on $M$ are elliptic
\begin{eqnarray*}
&&0\rightarrow A^0\rightarrow A^1\rightarrow A^2 \rightarrow A^3\rightarrow A_7^4\oplus A_{27}^4\rightarrow A_{14}^5\rightarrow 0,\\
&&0\rightarrow A_7^2\rightarrow A_1^3\oplus A_7^3\rightarrow A_1^4\rightarrow 0,\\
&&0\rightarrow A_1^3\oplus A_7^3\rightarrow A_1^4\oplus A_7^4\rightarrow 0.
\end{eqnarray*}
The first one is of type \uppercase\expandafter{\romannumeral1} generated by $\langle*\varphi\rangle.$ The last two are of type \uppercase\expandafter{\romannumeral2}. Moreover, if $\mathrm{Hol}(g)=G_2,$ except for the above three, $(\Lambda_{14}^2,D)$ and two complexes of type \uppercase\expandafter{\romannumeral3}, there is no other Dolbeault complexes on $M.$
 \end{theoremB}
 \begin{theoremC}
 Let $(M,g,\Omega)$ be a compact $\mathrm{Spin}(7)$-manifold. Then the following complexes on $M$ are elliptic
\begin{eqnarray*}
&&0\rightarrow A^0\rightarrow A^1\rightarrow A^2\rightarrow A^3\rightarrow A^\pm\rightarrow 0,\\
&&0\rightarrow A^0\rightarrow A^1\rightarrow A^2\rightarrow A_{48}^3\rightarrow A_{27}^4\rightarrow 0,\\
&&0\rightarrow A^0\rightarrow A^1\rightarrow A^2\rightarrow A^3\rightarrow A_1^4\oplus A_7^4\oplus A^-
\rightarrow A_8^5\rightarrow 0,\\
&&0\rightarrow A^0\rightarrow A^1\rightarrow A^2\rightarrow A^3\rightarrow A^4\rightarrow A^5\rightarrow A_{21}^6\rightarrow0,\\
&&0\rightarrow A_7^2\rightarrow A_8^3\rightarrow A_1^4\rightarrow0,\ \ 
0\rightarrow A_{21}^2\rightarrow A_{48}^3\rightarrow A_{27}^4\rightarrow0,\\
&&0\rightarrow A_8^3\rightarrow A_1^4\oplus A_7^4\rightarrow0,\ \ 
0\rightarrow A_1^4\oplus A_7^4\rightarrow A_8^5\rightarrow0.
\end{eqnarray*}
The first five are of type \uppercase\expandafter{\romannumeral1} generated by $\Lambda^\pm,\Lambda_8^3,\Lambda_{27}^4$ and $\Lambda_7^6$ respectively. The last four are of type \uppercase\expandafter{\romannumeral2}.
Morever, if $\mathrm{Hol}(g)=\mathrm{Spin}(7),$ then except for the above nine, $(\Lambda_{21}^2,D)$ and six complexes of type \uppercase\expandafter{\romannumeral3}, there is no other Dolbeault complexes on $M.$
 \end{theoremC}
 The key to proofs of Theorem B and Theorem C is based on an important fact that $G$ acts on the sphere $S^{n-1}$ transitively when $G$ is $G_2$ or $\mathrm{Spin}(7).$ See \cite{Borel1949} or \cite{MontSame1943}. However, these complexes cannot produce new geometric invariants. M. Fern\'{a}ndez and L. Ugarte \cite{Luis1998} computed the cohomology of the complex $(\Lambda_{14}^2,D).$ We show in \S \ref {lastsect77s} that, for every complex $C$ written as $0\rightarrow \Gamma(E_0)\rightarrow\Gamma(E_1)\rightarrow\cdots$ in Theorem B and Theorem C, the cohomology groups of $C$ can be described by harmonic forms, that is,
 $$H^i(C)\cong\{\alpha\in\Gamma(E_i)|\mathrm{d}\alpha=\mathrm{d}^*\alpha=0\}.$$
 \begin{remark}
By Wang's result \cite{Wang1989}, any $7$-manifold with a $G_2$-structure or $8$-manifold with a $\mathrm{Spin}(7)$-structure must be a spin manifold, and the associated metric is Ricci-flat. In this case, the spin bundle can be represented by appropriate $\Lambda_i^k.$ This forces that $D\alpha=0$ in Theorem B and Theorem C implies $\mathrm{d}\alpha=0.$ 
 \end{remark}
\section{ Linear algebra preliminaries}\label{firssectionjy}
\subsection{Prime forms in $\Lambda^{*}(\mathbb{R}^{n})$}
We begin by reviewing some facts about exterior algebra. Let $V$ be a real vector space and $\Lambda^{*}(V)$ the exterior algebra of $V$ which has a graded structure. If $V$ is equipped with a metric, there is a canonical isomorphism $i$ between $\Lambda^2(V)$ and the Lie algebra of the isometric transformation group of $V.$ For each $\alpha\in\Lambda^2(V),$ define the rank of $\alpha$ by $\mathrm{rank}\,\alpha:=\frac{1}{2}\mathrm{rank}\,i(\alpha).$ Note that the rank does not depend on metric of $V.$ In this section, we let $V=\mathbb{R}^n,n\geq 4$ be the Euclidean space.
\begin{definition}
We say a $k$-form $\alpha\in\Lambda^{k}(\mathbb{R}^{n})$ is \emph{prime} if the multiplication $\Lambda^1(\mathbb{R}^n)\rightarrow\Lambda^{k+1}(\mathbb{R}^n)$ by $\alpha$ is injective. Further, a linear subspace $E\subset\Lambda^{k}(\mathbb{R}^{n})$ is \emph{prime} if every non-zero form in $E$ is prime. 
\end{definition}
Obviously, $E\subset\Lambda^{k}(\mathbb{R}^{n})$ is prime if and only if the complex  
 \begin{eqnarray}\label{kprimcpm3}
 \Lambda^{k-2}(\mathbb{R}^{n})\stackrel{\lambda}{\longrightarrow}\Lambda^{k-1}(\mathbb{R}^{n})\stackrel{\lambda}{\longrightarrow}\Lambda^{k}(\mathbb{R}^{n})/E
 \end{eqnarray}
 is exact at $\Lambda^{k-1}(\mathbb{R}^{n})$ for all $\lambda\in\Lambda^{1}(\mathbb{R}^{n})-\{0\}.$
\begin{example}\label{lemzprimtwof2}
Consider the prime forms in $\Lambda^k(\mathbb{R}^n),n\geq 4$ for some special cases.
 \begin{itemize}
 \item  If $\alpha\in\Lambda^2(\mathbb{R}^n),$ then $\alpha$ is prime if and only if $\mathrm{rank}\,\alpha\geq 2.$
 \item When $k=n-2,$ every $\alpha\in\Lambda^{n-2}(\mathbb{R}^{n})$ can be expressed as 
 $$\alpha=\sum_{i,j}\alpha_{ij}\cdot e_1\cdots\hat{e}_i\cdots \hat{e}_j\cdots e_n$$
 with $\alpha_{ij}+(-1)^{i+j}\alpha_{ji}=0,$ where $\{e_1,\cdots,e_n\}$ is a basis for $\Lambda^1(\mathbb{R}^{n})$ and the hat denotes the omission of that element. Then $\alpha$ is prime if and only if the coefficient matrix $(\alpha_{ij})$ is invertible.
\item When $k=n-1,$ none of elements in $\Lambda^{n-1}(\mathbb{R}^{n})$ are prime. For each $\alpha\in \Lambda^{n-1}(\mathbb{R}^{n}),$ the multiplication by $\alpha$ is not injective since  
$$\mathrm{dim}\,\Lambda^{1}(\mathbb{R}^{n})=n>1=\mathrm{dim}\,\Lambda^{n}(\mathbb{R}^{n}).$$
\end{itemize}
 \end{example}
Let $*$ denote the Hodge star operator on $\Lambda^*(\mathbb{R}^n).$ Given a form $\alpha\in\Lambda^{n-2k}(\mathbb{R}^n),$ define a traceless operator $T_\alpha$ related to $\alpha$ by
$$T_\alpha:\ \Lambda^k(\mathbb{R}^n)\rightarrow\Lambda^k(\mathbb{R}^n),\ \ T_\alpha(\beta)=*(\alpha\cdot\beta).$$ 
The operator $T_\alpha$ is self-adjoint when $k$ is even and is anti-self-adjoint when $k$ is odd.
\begin{lemma}
Suppose $\alpha\in\Lambda^{n-2k}(\mathbb{R}^n).$ Then every eigenspace of $T_\alpha$ with non-zero eigenvalue is prime.
\end{lemma}
 \begin{proof}
Fix a non-zero eigenvalue $a$ of $T_\alpha,$ assume $T_\alpha(\beta)=a\cdot\beta$ and $\beta=\lambda\cdot\beta'$ for some $\beta'\in\Lambda^{k-1}(\mathbb{R}^n)$ and $\lambda\in\Lambda^1(\mathbb{R}^n)$ with $|\lambda|=1.$ Extend $\lambda$ to an orthonormal basis $B=\{\lambda,e_2,e_3,\cdots,e_{n}\}$ for $\Lambda^1(\mathbb{R}^n).$ We can assume that the expression of $\beta'$ with respect to $B$ does not contain $\lambda.$ Then the expression of $*(\alpha\cdot\lambda\cdot\beta')$ does not contain $\lambda.$ This is a contradiction to $T_\alpha(\beta)=a\cdot\beta.$ 
 \end{proof}
 \begin{example}\label{exar4pri65}
The eigenspaces $\Lambda^\pm\subset\Lambda^{2n}(\mathbb{R}^{4n})$ of star operator $*:\Lambda^{2n}(\mathbb{R}^{4n})\rightarrow\Lambda^{2n}(\mathbb{R}^{4n})$ associated with $\pm 1$ are both prime.
 \end{example}
 The basic action of $\mathrm{SO}(n)$ on $\mathbb{R}^{n}$ induces a representation on $\Lambda^{*}(\mathbb{R}^{n})$ by
$$A(\omega^{i_{1}}\wedge \omega^{i_{2}}\wedge\cdots \wedge \omega^{i_{k}})=A\omega^{i_{1}}\wedge A\omega^{i_{2}}\wedge\cdots \wedge A\omega^{i_{k}}.$$ 
In particular, $\Lambda^{2}(\mathbb{R}^{n})$ is the adjoint representation of $\mathrm{SO}(n)$ by identifying $\Lambda^{2}(\mathbb{R}^{n})$ with $\mathfrak{so}(n)$. The Hodge star $*:\Lambda^{k}\rightarrow\Lambda^{n-k}$ commutes with $\mathrm{SO}(n).$ Let $G$ be a subgroup of $\mathrm{SO}(n).$ The real representation $\mathbb{R}^n$ of $G$ induces an action of $G$ on the unit sphere $S^{n-1}$ in $\mathbb{R}^n.$ We give a simple but useful lemma.
 \begin{lemma}\label{usefullempri5}
 Let $E\subset\Lambda^k(\mathbb{R}^n)$ be a $G$-invariant subspace. If $G$ acts on $S^{n-1}$ transitively, then $E$ is prime if and only if the multiplication $E\rightarrow\Lambda^{k+1}(\mathbb{R}^n)$ by $\lambda$ is injective for some $\lambda\in\Lambda^1(\mathbb{R}^n).$
 \end{lemma}
 \subsection{Complete prime space and pair in $\Lambda^*(\mathbb{R}^n)$}
Sometimes we denote $\Lambda^i(\mathbb{R}^n)$ by $\Lambda^i.$ For a linear subspace $E$ of $\Lambda^*,$ let $E\cdot\Lambda^i$ denote the linear subspace of $\Lambda^{*+i}$ generated by $\{xy|\ x\in E,y\in\Lambda^i\}.$ Analogously to (\ref{kprimcpm3}), we introduce the notion of complete prime to describe the principal symbols of (\ref{firellty1}) and (\ref{firellty2}).
\begin{definition}
Let $E\subset\Lambda^{k}(\mathbb{R}^{n})$ and $F\subset\Lambda^{k-1}(\mathbb{R}^{n})$ be two linear subspaces for $k\geq 2.$ Set $E_i=E\cdot\Lambda^i$ for $i=1,2,\cdots,r$ and $E\cdot\Lambda^{r+1}=\Lambda^{k+r+1}.$
\begin{itemize}
\item[(1)] We say that $E$ is \emph{complete prime} if the complex 
 \begin{eqnarray}\label{comprimc5}
0\rightarrow\Lambda^0\stackrel{\lambda}{\rightarrow}\Lambda^1\stackrel{\lambda}{\rightarrow} \cdots\stackrel{\lambda}{\rightarrow}\Lambda^{k-1}\stackrel{\lambda}{\rightarrow}\Lambda^{k}/E\stackrel{\lambda}{\rightarrow}\Lambda^{k+1}/E_1\stackrel{\lambda}{\rightarrow}\cdots\stackrel{\lambda}{\rightarrow}\Lambda^{k+r}/E_r\rightarrow0 
 \end{eqnarray}
is exact for any $\lambda\in\Lambda^{1}-\{0\}.$ Denote this complex by $(E^\bullet,\lambda).$
\item[(2)] We say that the pair $(F,E)$ is \emph{complete prime} if the complex 
 \begin{eqnarray}\label{comprimc6}
0\rightarrow F\stackrel{\lambda}{\rightarrow}\Lambda^{k}/E\stackrel{\lambda}{\rightarrow}\Lambda^{k+1}/E_1\stackrel{\lambda}{\rightarrow}\Lambda^{k+2}/E_2\stackrel{\lambda}{\rightarrow}\cdots\stackrel{\lambda}{\rightarrow}\Lambda^{k+r}/E_r\rightarrow 0
 \end{eqnarray}
 is exact for any $\lambda\in\Lambda^{1}-\{0\}.$ Denote this complex by $(F,E^\bullet,\lambda).$
\end{itemize}
\end{definition}
\begin{remark}
The dual complex of (\ref{comprimc5}) is given by 
\begin{eqnarray}\label{comprimc8}
0\rightarrow E\stackrel{\lambda}{\rightarrow}E_1\stackrel{\lambda}{\rightarrow}E_2\stackrel{\lambda}{\rightarrow}\cdots\stackrel{\lambda}{\rightarrow}E_r\stackrel{\lambda}{\rightarrow}\Lambda^{k+r+1}\stackrel{\lambda}{\rightarrow}\cdots\stackrel{\lambda}{\rightarrow}\Lambda^n\rightarrow 0
\end{eqnarray}
which is the principal symbols of (\ref{firsellty3}). One can check that with the same generator  (\ref{comprimc5}) is exact if and only if (\ref{comprimc8}) is exact.
\end{remark}
\begin{example}
Consider $\mathbb{R}^{5}$ with an inner product. Fix a unit vector $u\in\mathbb{R}^{5}$ and let $W=\langle u\rangle^\perp.$ Then the eigenspaces $E=\Lambda^\pm(W)$ of star operator on $\Lambda^{2}(W)$ are both complete prime, i.e., the sequence $(E^\bullet,\lambda)$
\begin{eqnarray}\label{r5ellpcx}
0\rightarrow\Lambda^0\rightarrow\Lambda^1\rightarrow\Lambda^2/E\rightarrow\Lambda^3/(E\cdot\Lambda^1)\rightarrow 0
\end{eqnarray}
is exact. To see this, consider its alternating sum of dimensions. Due to $\mathrm{dim}\,E\cdot\Lambda^1=7,$ one can see that 
\begin{eqnarray}\label{dimforexam65}
\mathrm{dim}\,\Lambda^0-\mathrm{dim}\,\Lambda^1+(\mathrm{dim}\,\Lambda^2-\mathrm{dim}\,E)-(\mathrm{dim}\,\Lambda^3-\mathrm{dim}\,E\cdot\Lambda^1)=0.
\end{eqnarray}
Since $E$ is prime, (\ref{r5ellpcx}) is exact at $\Lambda^1.$ In addition, for any $\lambda\in\Lambda^1-\{0\}$ and $\alpha\in\Lambda^2,$ if $\lambda\cdot\alpha\in E\cdot\Lambda^1,$ then there exists $x\in E$ and $y\in\Lambda^1$ such that $\alpha=\lambda\cdot y+x.$ This means (\ref{r5ellpcx}) is exact at $\Lambda^2/E.$ It follows from (\ref{dimforexam65}) that (\ref{r5ellpcx}) is exact at $\Lambda^3/(E\cdot\Lambda^1).$
\end{example}
We next consider an equivalence relation between subspaces of exterior algebras. Suppose $V_1$ and $V_2$ are two linear spaces and $F_i\subset\Lambda^*(V_i),i=1,2$ are two linear subspaces. We say that $F_1$ is isomorphic to $F_2$ and denote by $F_1\cong F_2$ if there exists an isomorphism of algebras $ f:\Lambda^*(V_1)\rightarrow\Lambda^*(V_2)$ such that $F_2= f(F_1).$   
\begin{lemma}
Let $f$ be an isomorphism from $\Lambda^*(V)$ to $\Lambda^*(W).$ Then a subspace $E$ (resp. pair $(F,E)$) of $\Lambda^*(V)$ is complete prime if and only if $f(E)$ (resp. pair $(f(F),f(E))$) in $\Lambda^*(W)$ is complete prime.
\end{lemma}
\begin{proof}
The isomorphism $f$ induces a family of isomorphisms 
$$ f_i:\Lambda^{k+i}(V)/(E\cdot\Lambda^i(V))\rightarrow\Lambda^{k+i}(W)/(f(E)\cdot\Lambda^i(W)).$$ 
Then for every $\lambda\in\Lambda^1(V)-\{0\},$ we have an isomorphism between complexes 
$$\xymatrix{
&\cdots \ar[r]^-{\lambda}  &\Lambda^{k-1}(V)  \ar[d]_f \ar[r]^-{\lambda} &\Lambda^k(V)/E \ar[d]_{ f_0} \ar[r]^-{\lambda} &\Lambda^{k+1}(V)/(E\cdot\Lambda^1(V)) \ar[d]_{ f_1} \ar[r]^-{\lambda} &\cdots\\
 &\cdots \ar[r]^-{ f(\lambda)} &\Lambda^{k-1}(W) \ar[r]^-{ f(\lambda)} &\Lambda^k(W)/ f(E) \ar[r]^-{ f(\lambda)} &\Lambda^{k+1}(W)/( f(E)\cdot\Lambda^1(W)) \ar[r]^-{ f(\lambda)} &\cdots.
}$$
The two families of homology groups are isomorphism respectively. Hence, $(E^\bullet,\lambda)$ is exact if and only if $( f(E)^{\bullet}, f(\lambda))$ is exact. Same proof for the pair $(E,F).$
\end{proof}
\begin{corollary}\label{coroimpcp}
Let $G$ be a subgroup of $\mathrm{SO}(n)$ and let $E,F$ be two $G$-invariant subspaces of $\Lambda^*(\mathbb{R}^n).$ If $G$ acts on $S^{n-1}$ transitively, then $E$ (resp. $(F,E)$) is complete prime if and only if $(E^\bullet,\lambda)$ (reps. $(F,E^\bullet,\lambda)$) is exact for some $\lambda\in\Lambda^1(\mathbb{R}^n)-\{0\}.$ 
\end{corollary}
\begin{definition}
Let $M$ be a smooth manifold.  A subbundle $E$ of $\Lambda^{k}(M)$ is \emph{perfect} if $E_x\cong E_y$ for all $x,y\in M.$ Further, $E$ is prime (resp. complete prime) if $E$ is perfect and $E_x\subset\Lambda^{k}(T_x^*M)$ is prime (resp. complete prime) for some $x\in M.$ A bundle pair $(F,E)$ is complete prime if $E,F$ are both perfect and $(F_x,E_x)$ is complete prime for some $x\in M.$
\end{definition}
\begin{remark}
In general, the rank of $E\cdot\Lambda^r$ depends on the algebraic structure of $E.$ If $E$ is perfect, then $E\cdot\Lambda^r(M)$ is also perfect. As a result, $E\cdot\Lambda^r(M)$ is a subbundle of $\Lambda^*(M).$
\end{remark}
\section{Dolbeault-type complexes}\label{pdrcdef5}
Let $(M,g)$ be a Riemannian manifold with a metric $g.$ Suppose that $E$ is a subbundle of $\Lambda^{k}(M),2\leq k\leq n-2.$ There is an orthogonal decomposition $\Lambda^{k}(M)=E\oplus E^{\perp}$ with respect to $g.$ Let $p$ be the projection from $\Lambda^{k}(M)$ to $E^{\perp}.$ For the ease of notations, we use $\Lambda^k$ for $\Lambda^k(M)$ in this section. We will seek the elliptic conditions for complexes (\ref{firellty1}) and (\ref{firellty2}). 

Let $G$ be a subgroup of $\mathrm{SO}(n)$ with a Lie algebra $\mathfrak{g}.$ Recall from \cite{Chern1966} that a $G$-structure $Q$ on $M$ is integrable if it is locally isomorphic to a flat $G$-structure, i.e., for every point $x\in M,$ there is an open set $U$ of $x$ such that $Q|_U$ has a local section $(V_1,\cdots,V_n)$ consisting of commuting vector fields. Obviously, if $Q$ is integrable then the associated bundle $\Lambda(\mathfrak{g}):=Q\times_G\mathfrak{g}$ is locally flat, i.e., for every $x\in M,$ there exists an open set $U$ of $x$ and closed forms $\alpha_1,\cdots,\alpha_s$ as a $C^\infty(U)$-basis for $\Gamma(U,\Lambda(\mathfrak{g})).$ We introduce a  concept of locally quasi-flat which is weaker than locally flat.
\begin{definition}
We say that a subbundle $F$ of $\Lambda^k(M)$ is \emph{locally quasi-flat} if 
$\mathrm{d}\Gamma(M,F)\subset\Gamma(M,F\cdot\Lambda^1).$
\end{definition}
As a direct consequence of definition, we have 
\begin{lemma}\label{requafla2}
If $F\subset\Lambda^k(M)$ is locally quasi-flat, then $F\cdot\Lambda^1$ is also locally quasi-flat. 
\end{lemma}
Let $E$ be a locally quasi-flat subbundle of $\Lambda^k(M)$ and we define a list of subbundles $E_i\subset\Lambda^{k+i}$ by $E_i=E\cdot\Lambda^i.$ Lemma \ref{requafla2} shows that every $E_i$ is locally quasi-flat. We claim that the differential in (\ref{firellty1}) and (\ref{firellty2}) satisfy $D_{i+1}D_i=0.$ To see this, for $\alpha\in\Gamma(M,E_{i-1}^\perp),$ set $\mathrm{d}\alpha=u_1+u_2$ with $u_1\in\Gamma(M,E_{i})$ and $u_2\in\Gamma(M,E_{i}^\perp).$ Using $\mathrm{d}u_1+\mathrm{d}u_2=0,$ we have 
$\mathrm{d}u_2=-\mathrm{d}u_1\in\Gamma(M,E_i\cdot\Lambda^1)=\Gamma(M,E_{i+1}).$ Then 
$D_{i+1}D_i\alpha=p_{i+1}\mathrm{d}p_i \mathrm{d}\alpha=p_{i+1}\mathrm{d}u_2=0.$ 

A simple calculation shows that (\ref{comprimc5}) and (\ref{comprimc6}) are principal symbols of  (\ref{firellty1}) and (\ref{firellty2}) respectively. According to the elliptic complex theory \cite{AtiyahBott1967}, a complex is elliptic if and only if the sequence of its principal symbols is exact outside the zero section. So far we obtain the following theorem.
\begin{theorem}\label{theo2basicec}
Suppose $M$ is a closed Riemannian manifold, $E$ is a locally quasi-flat subbundle of $\Lambda^k(M)$ and $F$ is a subbundle of $\Lambda^{k-1}(M).$ If $E$ is complete prime, then $(E,D)$ is an elliptic complex. Similarly, if $(F,E)$ is complete prime, then $(F,E,D)$ is an elliptic complex.
\end{theorem}
Specially, we have 
\begin{corollary}\label{theoellipcon}
Let $M$ be a closed Riemannian $n$-manifold and $E$ a subbundle of $\Lambda^k(M).$ Then the complex 
\begin{eqnarray}\label{sderham03}
0\longrightarrow A^{0}(M)\stackrel{\mathrm{d}}{\longrightarrow}A^{1}(M)\stackrel{\mathrm{d}}{\longrightarrow}\cdots\stackrel{\mathrm{d}}{\longrightarrow}A^{k-1}(M)\stackrel{p\mathrm{d}}{\longrightarrow}\Gamma(M,E^{\perp})\longrightarrow 0
\end{eqnarray}
 is elliptic if and only if $\mathrm{rank}\,E=\sum_{i=0}^{k}(-1)^{k+i}\binom{n}{i}$ and $E$ is prime.
\end{corollary}
\begin{proof}
Fix a point $x\in M.$ For any $\lambda\in\Lambda^1(T_x^*M)-\{0\},$ consider the complex 
\begin{eqnarray}\label{spec56f}
0\longrightarrow \Lambda^0\stackrel{\lambda}{\longrightarrow}\Lambda^1\stackrel{\lambda}{\longrightarrow}\cdots\stackrel{\lambda}{\longrightarrow}\Lambda^{k-1}\stackrel{p(\lambda\cdot\ )}{\longrightarrow}E^\perp\rightarrow 0.
\end{eqnarray}
It is obvious that (\ref{spec56f}) is exact at $\Lambda^i$ for $i\leq k-2.$ By (\ref{kprimcpm3}) we know that (\ref{spec56f}) is exact at $\Lambda^{k-1}$ and $E^\perp$ for all $\lambda$ if and only if $E$ is prime and 
$$(-1)^{k}\mathrm{rank}\,E^{\perp}+\sum_{i=0}^{k-1}(-1)^{i}\mathrm{rank}\,\Lambda^{i}(M)=0.$$
\end{proof}
\begin{example}\label{examplesign}
Consider a closed, oriented Riemannian $4n$-manifold $(M,g)$. The Hodge star $*$ operator induces a splitting of the space of $2n$-forms $\Lambda^{2n}=\Lambda^+\oplus\Lambda^-$ into the subspace of self-dual forms $A^+$ and anti-self-dual forms $A^-.$ According to Example \ref{exar4pri65}, $\Lambda^+$ and $\Lambda^-$ are both prime. In addition, 
$$\mathrm{rank}\,\Lambda^\pm=\frac{1}{2}\binom{4n}{2n}=\sum_{i=0}^{2n}(-1)^i\binom{4n}{i}.$$  Apply Corollary \ref{theoellipcon}, the complex
 $$C_\pm:A^{0}(M)\stackrel{\mathrm{d}}{\longrightarrow}A^{1}(M)\stackrel{\mathrm{d}}{\longrightarrow}\cdots\stackrel{\mathrm{d}}{\longrightarrow}A^{2n-1}(M)\stackrel{p\mathrm{d}}{\longrightarrow}A^\pm\rightarrow 0$$ 
 generated by $\Lambda^\mp$ is elliptic. Obviously, the $2n$-th cohomology of $C_\pm$ is $H^\pm(M).$ Since $\mathrm{d}\alpha\in A^\pm$ implies $\mathrm{d}\alpha=0,$ the $(2n-1)$-th cohomology of $C_\pm$ is $H^{2n-1}(M).$
\end{example}

\section{Complexes and representations}\label{lastsectcr}
Inspired by D. Salamon \cite{Salamon1989} and Reyes-Carrion \cite{Reyes1998}, we use representations of holonomy group to construct complete prime bundles. 

Let $G$ be a subgroup of $\mathrm{SO}(n)$ and $Q$ a $G$-structure on a closed $n$-manifold $M.$ Such a $G$-structure induces an orientation of $M$ and a Riemannian metric $g$ with holonomy group $\mathrm{Hol}(g)$ contained in $G.$ The basic action of $G$ on $\mathbb{R}^n$ induces a representation of $G$ on $\Lambda^*(\mathbb{R}^n).$ We write $\Lambda^k(\mathbb{R}^n)=\oplus_{i\in I_k}W_i^k,$ where each $W_i^k\subset\Lambda^k$ is an irreducible representation of $G$ and $I_k$ is a finite indexing set. The Hodge star $*$ gives an isometry between $W_i^k$ and $W_i^{n-k}.$ Let $\Lambda_i^k=Q\times_G W_i^k$ be the subbundle of $\Lambda^k(M)$ corresponding to $W_i^k.$ Then there is a natural decomposition \cite[Proposition 3.5.1]{Joyce2007}
\begin{eqnarray}\label{decombunp}
\Lambda^*(M)=\bigoplus_{k=0}^n\bigoplus_{i\in I_k}\Lambda_i^k.
\end{eqnarray}
We let $\pi_i$ denote the orthogonal projection from $\Lambda^k$ to $\Lambda_i^k,$ and let $\pi_{i,j}$ denote the orthogonal projection from $\Lambda^k$ to $\Lambda_i^k\oplus\Lambda_j^k.$ This notation will be used throughout the rest of this paper.

 In fact, each $\Lambda_i^k$ can be constructed by parallel translation. Let $\mathcal{E}$ be a $G$-invariant subspace of $\Lambda^k(\mathbb{R}^n).$ The vector bundle $\Lambda^*(M)$ carries a connection $\nabla$ induced by the Levi-Civita connection of $g.$ Fix a base point $x\in M$ and identify $T_x^*M$ with $\mathbb{R}^n.$
For each $y\in M$ and every path $\gamma:[0,1]\rightarrow M$ with $\gamma(0)=x,\gamma(1)=y,$ the parallel transport along $\gamma$ defines an isomorphism $P_\gamma:\Lambda^*(T_x^*M)\rightarrow\Lambda^*(T_y^*M).$ We claim that the subspace $P_\gamma(\mathcal{E})\subset\Lambda^k(T_y^*M)$ is independent of $\gamma.$ To see this, consider another path $\gamma'\in C^\infty([0,1],M)$ connecting $x$ and $y,$ the loop $\gamma'^{-1}\circ\gamma$ through $x$ determines a transformation $P_{\gamma'}^{-1}\circ P_\gamma\in\mathrm{Hol}(g)$ and hence $P_{\gamma'}^{-1}\circ P_\gamma(\mathcal{E})=\mathcal{E}$ which means $P_\gamma(\mathcal{E})=P_{\gamma'}(\mathcal{E}).$ Then we get a vector bundle $E=\bigcup_{y\in M}E_y,$ where $E_y=P_\gamma(\mathcal{E}).$

 Let $\{e_i\}$ be an orthogonal frame field on $M$ and $\{\omega^i\}$ the dual orthogonal frame field. The differential can be written as $\mathrm{d}=\sum_{i=1}^n \omega^i\wedge\nabla_{e_i}.$ It is important to note that $\nabla$ can reduce to a connection on $E$ since $\mathrm{Hol}(g)$ preserves $E.$ That is, for all $\xi\in\Gamma(E),$ we have $\nabla\xi\in\Gamma(E\otimes\Lambda^1).$ This implies that $E$ is locally quasi-flat. Besides, since parallel transport along a path induces an isomorphism between algebras, then $E$ (resp. $(F,E)$) is complete prime if and only if $\mathcal{E}$ (resp. $(\mathcal{F},\mathcal{E})$) is complete prime. Using Theorem \ref{theo2basicec}, we have the following theorem
 \begin{theorem}\label{theA}
Let $G$ a connected subgroup of $\mathrm{SO}(n)$ and let $(M,g,G,Q)$ be a closed Riemannian $n$-manifold with a compatible torsion-free $G$-structure $Q.$ Suppose $\mathcal{F}\subset\Lambda^{k-1}(\mathbb{R}^n)$ and $ \mathcal{E}\subset\Lambda^k(\mathbb{R}^n)$ are two invariant subspaces under the action of $G.$ Set $F=Q\times_G\mathcal{F}\subset\Lambda^{k-1}(M)$ and $E=Q\times_G\mathcal{E}\subset\Lambda^k(M).$ Then $E$ is locally quasi-flat. Consequently,
\begin{itemize}
\item[(\romannumeral1)] If $\mathcal{E}$ is complete prime, then the sequence $(E,D)$ is an elliptic complex. 
\item[(\romannumeral2)] If $(\mathcal{F},\mathcal{E})$ is complete prime, then the sequence $(F,E,D)$ is an elliptic complex. 
 \end{itemize}
\end{theorem}

We next discuss two cases $G=G_2$ and $G=\mathrm{Spin}(7).$ For more details on compact manifolds with exceptional holonomy, see \cite{Salamon1989} and \cite{Joyce2000}.
\subsection{Dolbeault complexes on $G_2$-manifold}
Let $(x_1,x_2,\cdots,x_7)$ be coordinates on $\mathbb{R}^7.$ For simplicity write $\omega^{ij\cdots l}$ to denote the product $\mathrm{d}x_i\wedge \mathrm{d}x_j\wedge\cdots\wedge \mathrm{d}x_l.$ Define a $G_2$-invariant $3$-form on $\mathbb{R}^7$ by 
$$\varphi_0=\omega^{123}+\omega^{145}+\omega^{167}+\omega^{246}-\omega^{257}-\omega^{347}-\omega^{356}.$$
It follows from \cite[Lemma 11.4]{Salamon1989} that the representation $\Lambda^*(\mathbb{R}^7)$ of $G_2$ splits orthogonally into components as 
\begin{eqnarray*}
&&\Lambda^1=\Lambda_7^1,\ \ \Lambda^2=\Lambda_7^2\oplus\Lambda_{14}^2,\ \ \Lambda^3=\Lambda_1^3\oplus\Lambda_7^3\oplus\Lambda_{27}^3,\\
&& \Lambda^4=\Lambda_1^4\oplus\Lambda_7^4\oplus\Lambda_{27}^4,\ \ \Lambda^5=\Lambda_7^5\oplus\Lambda_{14}^5,\ \ \Lambda^6=\Lambda_7^6.
\end{eqnarray*}
See also \cite{Bryant1987,DJoyce1996,Luis1998}. The prime forms $\varphi_0$ and $*\varphi_0$ induce canonical isomorphisms between $\Lambda_7^1,\Lambda_7^2,\cdots,\Lambda_7^6$ given by 
\begin{eqnarray}\label{canoisog27}
\Lambda_7^2=*(*\varphi_0\cdot\Lambda^1),\ \Lambda_7^3=*(\varphi_0\Lambda^1),\  \Lambda_7^4=\varphi_0\Lambda^1,\ \Lambda_7^5=*\varphi_0\cdot\Lambda^1,\ \Lambda_7^6=*\Lambda^1.
\end{eqnarray}
Moreover, $\Lambda_7^2$ and $\Lambda_{14}^2$ are the eigenspaces of the operator $T_{\varphi_0}=*(\varphi_0\wedge\cdot\ ):\Lambda^2\rightarrow\Lambda^2$ with eigenvalues $2$ and $-1$ respectively.
\begin{proposition}\label{g2relat355}
These spaces $\Lambda_i^k$ satisfy the following relations
\begin{eqnarray*}
&&\Lambda_7^2\cdot\Lambda^1=\Lambda^3,\ \ \Lambda_{14}^2\cdot\Lambda^1=\Lambda_{7}^3\oplus\Lambda_{27}^3,\\
&&\Lambda_7^3\cdot\Lambda^1=\Lambda^4,\ \ \Lambda_{27}^3\cdot\Lambda^1=\Lambda_7^4\oplus\Lambda_{27}^4,\\
&&\Lambda_7^4\cdot\Lambda^1=\Lambda_{27}^4\cdot\Lambda^1=\Lambda^5,\\ 
&&\Lambda_7^5\cdot\Lambda^1=\Lambda_{14}^5\cdot\Lambda^1=\Lambda_1^3\cdot\Lambda_7^3=\Lambda_1^4\cdot\Lambda_7^2=\Lambda^6,\\
&&\Lambda_1^3\cdot\Lambda_7^2=\Lambda_7^5,\ \ \Lambda_1^3\cdot\Lambda_{14}^2=\Lambda_{14}^5,\\
&&\Lambda_1^4\cdot\Lambda_{14}^2=\Lambda_1^3\cdot\Lambda_1^3=\Lambda_1^3\cdot\Lambda_{27}^3=0.
\end{eqnarray*}
\end{proposition}
We introduce an important lemma about $G_2.$ See \cite{Borel1949} or \cite{MontSame1943}. 
\begin{lemma}[{\cite[Theorem \uppercase\expandafter{\romannumeral3}]{Borel1949}}]\label{lemtransact}
Let $\mathbb{R}^7$ be the basic representation of $G_2.$ Then $G_2$ acts on $S^6$ transitively.
\end{lemma}
Using Lemma \ref{lemtransact} and Lemma \ref{usefullempri5}, we have
\begin{proposition}\label{propprisps6}
All $G_2$-invariant prime subspaces in $\Lambda^*(\mathbb{R}^7)$ are as follows 
$$\Lambda_7^2,\ \ \Lambda_{14}^2,\ \ \Lambda_1^3,\ \ \Lambda_7^3,\ \ \Lambda_1^3\oplus\Lambda_7^3,\ \ \Lambda_1^4.$$ 
\end{proposition}
If $(M,g)$ is a Riemannian $7$-fold with holonomy contained in $G_2,$ there is a torsion-free $G_2$-structure $Q$ and a parallel $3$-form $\varphi$ induced by $\varphi_0.$ By \cite[Lemma 11.5]{Salamon1989}, the torsion-free condition for $\varphi$ is equivalent to the condition $\mathrm{d}\varphi=\mathrm{d}^*\varphi=0.$ Reyes-Carrion showed in \cite{Reyes1998} that the complex 
\begin{eqnarray}\label{g2ellipticom0123}
0\rightarrow A^0\rightarrow A^1\rightarrow A_7^2\rightarrow A_1^3\rightarrow 0
\end{eqnarray}
generated by the adjoint representation of $G_2$ is elliptic. We develop the following theorem that gives all Dolbeault complexes on a compact $G_2$-manifold generated by $G_2$-invariant subspaces of $\Lambda^*(\mathbb{R}^7).$  
 \begin{theorem}\label{g2ellpcthe}
Let $(M,g,\varphi)$ be a compact $G_2$-manifold. Then the following complexes on $M$ are elliptic
\begin{eqnarray}
&&0\rightarrow A^0\rightarrow A^1\rightarrow A^2 \rightarrow A^3\rightarrow A_7^4\oplus A_{27}^4\rightarrow A_{14}^5\rightarrow 0,\label{g2ellipticom14} \\
&&0\rightarrow A_7^2\rightarrow A_1^3\oplus A_7^3\rightarrow A_1^4\rightarrow 0,\label{g2ellipticom234}\\
&&0\rightarrow A_1^3\oplus A_7^3\rightarrow A_1^4\oplus A_7^4\rightarrow 0.\label{g2ellipticom34}
\end{eqnarray}
The first one is of type \uppercase\expandafter{\romannumeral1} generated by $\langle*\varphi\rangle.$ The last two are of type \uppercase\expandafter{\romannumeral2}. Moreover, if $\mathrm{Hol}(g)=G_2,$ except for the above three, (\ref{g2ellipticom0123}) and two dual complexes of (\ref{g2ellipticom0123}) and (\ref{g2ellipticom14}), there is no other Dolbeault complexes on $M.$
 \end{theorem}
\begin{proof}
We divide the proof into three parts. By (\ref{canoisog27}), let $\{e_i^k\}_{i=1}^7$ be a basis for $\Lambda_7^k,k=2,3,4,5,$ where $e_i^2=*(*\varphi_0\cdot\omega^i),e_i^3=*(\varphi_0\cdot\omega^i),e_i^4=\varphi_0\cdot\omega^i,e_i^5=(*\varphi_0)\cdot\omega^i.$ 
\begin{itemize}
\item[(1)]  According to Theorem \ref{theA}, (\ref{g2ellipticom14}) is elliptic if we show that the line bundle $\langle*\varphi\rangle$ is complete prime. Using Lemma \ref{lemtransact} and Corollary \ref{coroimpcp}, it suffices to show that the symbol sequence of (\ref{g2ellipticom14}) is exact for some $\lambda\in\Lambda^1-\{0\}.$ This is equivalent to that $\pi_{14}(\lambda\cdot):\Lambda_7^4\oplus\Lambda_{27}^4\rightarrow\Lambda_{14}^5$ is surjective for some $\lambda\in\Lambda^1-\{0\},$ and also is equivalent to $\mathrm{dim}\,\lambda\Lambda^4\cap\Lambda_7^5=1$ since $\lambda (\Lambda_7^4\oplus\Lambda_{27}^4)=\lambda\Lambda^4.$ Set $\psi=*\varphi_0.$

\textbf{Claim:} Take $\lambda=\omega^1,$ we have $\omega^1\Lambda^4\cap\Lambda_7^5=\langle\omega^1\psi\rangle.$\\
Assume $\beta=\sum_I b_I\omega^I\in\Lambda^4$ such that $\omega^1\beta\in\Lambda_7^5,$ where the index $I$ does not contain $1.$ The basis $\{e_i^5\}_{i=1}^7$ for $\Lambda_7^5$ is given by
\begin{eqnarray*}
&&e_1^5=\omega^{14567}+\omega^{12367}+\omega^{12345},\ \ 
e_2^5=\omega^{24567}-\omega^{12357}+\omega^{12346},\\
&&e_3^5=\omega^{34567}-\omega^{12356}-\omega^{12347},\ \ 
e_4^5=\omega^{23467}+\omega^{13457}-\omega^{12456},\\
&&e_5^5=\omega^{23567}+\omega^{13456}-\omega^{12457},\ \ 
e_6^5=\omega^{23456}-\omega^{13567}+\omega^{12467},\\
&&e_7^5=\omega^{23457}-\omega^{13467}-\omega^{12567}.
\end{eqnarray*}
Set $\omega^1\beta=\sum_{i=1}^7a_ie_i^5.$ We can see immediately that $a_i=0$ for $i\geq 2.$ Hence, 
$$\omega^1\beta=a_1(\omega^{14567}+\omega^{12367}+\omega^{12345})=a_1\omega^1\psi.$$
\item[(2)] We next prove that the two complexes 
\begin{eqnarray}
&&0\longrightarrow\Lambda_7^2\stackrel{\pi_{1,7}\circ\lambda}{\longrightarrow}\Lambda_1^3\oplus\Lambda_7^3\stackrel{\pi_{1}\circ\lambda}{\longrightarrow}\Lambda_1^4\longrightarrow0,\label{firsprfthmb}\\ 
&&0\longrightarrow\Lambda_1^3\oplus\Lambda_7^3\stackrel{\pi_{1,7}\circ\lambda}{\longrightarrow}\Lambda_1^4\oplus\Lambda_7^4\longrightarrow0\label{secsprfthmb}
\end{eqnarray}
 are exact when $\lambda=\omega^1.$ By calculation we have 
$$\pi_{1,7}(e_1^2\cdot\omega^1)=\frac{3}{7}\varphi_0,\ \pi_{1,7}(e_{2i}^2\cdot\omega^1)=\frac{1}{2}e_{2i+1}^3,\ \pi_{1,7}(e_{2i+1}^2\cdot\omega^1)=-\frac{1}{2}e_{2i}^3$$
 for $i=1,2,3.$ Then $\pi_{1,7}\circ\omega^1$ is injective. It follows from $\pi_{1}(e_1^3\cdot\omega^1)=-\frac{4}{7}*\varphi_0$ that $\pi_{1}\circ\omega^1$ is surjective. Hence 
(\ref{firsprfthmb}) is exact. Similarly, (\ref{secsprfthmb}) is exact since 
$$\pi_{1,7}(\varphi_0\cdot\omega^1)=e_1^4,\pi_{1,7}(e_1^3\cdot\omega^1)=-\frac{4}{7}*\varphi_0, \pi_{1,7}(e_{2i}^3\cdot\omega^1)=-\frac{1}{2}e_{2i+1}^4, \pi_{1,7}(e_{2i+1}^3\cdot\omega^1)=\frac{1}{2}e_{2i}^4$$
 for $i=1,2,3.$
\item[(3)] When $\mathrm{hol}(g)=G_2,$ we explain that (\ref{g2ellipticom0123}), (\ref{g2ellipticom14}), (\ref{g2ellipticom234}) and (\ref{g2ellipticom34}) are the only Dolbeault complexes of type \uppercase\expandafter{\romannumeral1} and \uppercase\expandafter{\romannumeral2} on $M.$ Notice that a complete prime space must be prime and a complete prime pair $(F,E)$ implies $F$ is prime. From Propostion \ref{propprisps6}, there are only six $G_2$-invariant prime spaces. According to Proposition \ref{g2relat355}, one can check one by one that $\Lambda_7^2,\Lambda_1^3,\Lambda_7^3,\Lambda_1^3\oplus\Lambda_7^3$ are not complete prime spaces. In addition, except for $(\Lambda_7^2,\Lambda_{27}^3)$ and $(\Lambda_1^3\oplus\Lambda_7^3,\Lambda_{27}^4),$ the rest of possible complexes of type \uppercase\expandafter{\romannumeral2} are $\pi_1\lambda:\Lambda_1^3\rightarrow\Lambda_1^4$ and $\pi_7\lambda:\Lambda_7^3\rightarrow\Lambda_7^4.$ But, they are impossible since $\pi_1\lambda=0$ and $\pi_7(e_i^3\cdot\omega^i)=0.$ This completes the proof.
\end{itemize}
\end{proof}
\subsection*{Dolbeault complexes on $\mathrm{Spin}(7)$-manifold}
We review from \cite{Joyce1996,Joyce2000} some facts about $\mathrm{Spin}(7)$-structure. Let $(x_1,\cdots,x_8)$ be coordinates of $\mathbb{R}^8.$ Define a closed, self-dual, $\mathrm{Spin}(7)$-invariant
$4$-form $\Omega_0$ on $\mathbb{R}^8$ by 
\begin{equation}\label{Omega0def}
\begin{split}
\Omega_0&=\omega^{1234}+\omega^{1256}+\omega^{1278}+\omega^{1357}-\omega^{1368}-\omega^{1458}-\omega^{1467}\\
&-\omega^{2358}-\omega^{2367}-\omega^{2457}+\omega^{2468}+\omega^{3456}+\omega^{3478}+\omega^{5678},
\end{split}
\end{equation}
where $\omega^{ij\cdots l}=\mathrm{d}x_i\wedge \mathrm{d}x_j\wedge\cdots\wedge \mathrm{d}x_l.$ As a representation of $\mathrm{Spin}(7),$ the exterior algebra $\Lambda^*(\mathbb{R}^8)$ splits orthogonally into irreducible representations (ref.  \cite{Salamon1989,Joyce1996})
\begin{eqnarray*}
&&\Lambda^1=\Lambda_8^1,\ \  \Lambda^2=\Lambda_7^2\oplus\Lambda_{21}^2,\ \ \Lambda^3=\Lambda_8^3\oplus\Lambda_{48}^3,\\
 &&\Lambda^4=\Lambda_1^4\oplus\Lambda_7^4\oplus\Lambda_{27}^4\oplus\Lambda_{35}^4,\ \ \Lambda^-=\Lambda_{35}^4, \\
 &&\Lambda^5=\Lambda_8^5\oplus\Lambda_{48}^5,\ \ \Lambda^6=\Lambda_7^6\oplus\Lambda_{21}^6,\ \ \Lambda^7=\Lambda_8^7.
 \end{eqnarray*}
We note that $\Lambda_{7}^{2}$ and $\Lambda_{21}^2$ are the eigenspaces of the operator $T_{\Omega_0}:\Lambda^2\rightarrow\Lambda^2$ with eigenvalues $-1$ and $3,$ respectively. Moreover, there are canonical isomorphisms
\begin{eqnarray}\label{isospin7lambda7i}
\Lambda_8^1\cong\Lambda_8^3\cong\Lambda_8^5\cong\Lambda_8^7,\ \ \Lambda_7^2\cong\Lambda_7^4\cong\Lambda_7^6
\end{eqnarray}
 by a multiplication of $\Omega_0.$ From \cite[Proposition 12.5]{Salamon1989}, the subspace $\Lambda_7^4$ can be described in the following way. Write $\mathbb{R}^8=\mathbb{R}\oplus\mathbb{R}^7$ and let $\Lambda_7^3(\mathbb{R}^7)$ be the $G_2$-invariant subspace.  Define a map $i:\Lambda_7^3(\mathbb{R}^7)\rightarrow\Lambda^+(\mathbb{R}^8)$ by $i(\alpha)=\omega^1\wedge\alpha+*_7\alpha.$ Then $i$ is injective and we have $\Lambda_7^4(\mathbb{R}^8)=\mathrm{image}\,i.$

As a subgroup of $\mathrm{Spin}(7),$ the special unitary group $\mathrm{SU}(4)$ acts as usual on the sphere $S^7$ (in $\mathbb{R}^8=\mathbb{C}^4$) transitively. Then $\mathrm{Spin}(7)$ acts transitively on $S^7.$ Using Lemma \ref{usefullempri5}, we have
\begin{proposition}\label{allprimepin7d}
All $\mathrm{Spin}(7)$-invariant prime subspaces of $\Lambda^*(\mathbb{R}^8)$ are as follows 
$$\Lambda_7^2,\ \Lambda_{21}^2,\ \Lambda_8^3,\ \Lambda^-,\ \Lambda_7^6,\ \mbox{seven combinations of }\Lambda_1^4,\Lambda_7^4 \mbox{ and }\Lambda_{27}^4.$$
\end{proposition}
\begin{proposition}\label{lakirela6}
These spaces $\Lambda_i^k$ satisfy the following relations
\begin{eqnarray*}
&&\Lambda_7^2\cdot\Lambda^1=\Lambda_{21}^2\cdot\Lambda^1=\Lambda^3,\\  &&\Lambda_8^3\cdot\Lambda^1=(\Lambda_{27}^4)^\perp,\ \ \Lambda_{48}^3\cdot\Lambda^1=(\Lambda_1^4)^\perp,\\ 
&&\Lambda_7^4\cdot\Lambda^1=\Lambda^\pm\cdot\Lambda^1=\Lambda^5,\ \ \Lambda_{27}^4\cdot\Lambda^1=\Lambda_{48}^5,\\
&&\Lambda_8^5\cdot\Lambda^1=\Lambda_{48}^5\cdot\Lambda^1=\Lambda^6.
\end{eqnarray*}
\end{proposition}
Reyes-Carrion showed in \cite{Reyes1998} that the complex 
$$C_1:0\rightarrow\Lambda^0\rightarrow\Lambda^1\rightarrow\Lambda_7^2\rightarrow 0$$
generated by the adjoint representation of $\mathrm{Spin}(7)$ is elliptic. The following theorem gives all Dolbeault complexes on a compact $\mathrm{Spin}(7)$-manifold generated by $\mathrm{Spin}(7)$-invariant subspaces of $\Lambda^*(\mathbb{R}^8).$ We let $\widehat{C}=A^*(M)/C$ denote the dual complex of $C.$
 \begin{theorem}\label{spin7ellpcthe}
 Let $(M,g,\Omega)$ be a compact $\mathrm{Spin}(7)$-manifold. Then the following complexes on $M$ are elliptic
\begin{eqnarray*}
C_\pm&:& 0\rightarrow A^0\rightarrow A^1\rightarrow A^2\rightarrow A^3\rightarrow A^\pm\rightarrow 0,\\
C_1\oplus C_3&:& 0\rightarrow A^0\rightarrow A^1\rightarrow A^2\rightarrow A_{48}^3\rightarrow A_{27}^4\rightarrow 0,\\
C_-\oplus C_5&:& 0\rightarrow A^0\rightarrow A^1\rightarrow A^2\rightarrow A^3\rightarrow A_1^4\oplus A_7^4\oplus A^-\rightarrow A_8^5\rightarrow 0,\\
\widehat{C}_6&:& 0\rightarrow A^0\rightarrow A^1\rightarrow A^2\rightarrow A^3\rightarrow A^4\rightarrow A^5\rightarrow A_{21}^6\rightarrow0,\\
C_2&:& 0\rightarrow A_7^2\rightarrow A_8^3\rightarrow A_1^4\rightarrow0,\\
C_3&:& 0\rightarrow A_{21}^2\rightarrow A_{48}^3\rightarrow A_{27}^4\rightarrow0,\\
C_4&:& 0\rightarrow A_8^3\rightarrow A_1^4\oplus A_7^4\rightarrow0,\\
C_5&:& 0\rightarrow A_1^4\oplus A_7^4\rightarrow A_8^5\rightarrow0,\\
C_6&:&  0\rightarrow A_7^6\rightarrow A^7\rightarrow A^8\rightarrow0.
\end{eqnarray*}
The first five are of type \uppercase\expandafter{\romannumeral1} generated by $\Lambda^\pm,\Lambda_8^3,\Lambda_{27}^4$ and $\Lambda_7^6$ respectively. The last four are of type \uppercase\expandafter{\romannumeral2}.
Morever, if $\mathrm{Hol}(g)=\mathrm{Spin}(7),$ then except for the above nine, $C_1$ and six complexes of type \uppercase\expandafter{\romannumeral3}, there is no other Dolbeault complexes on $M.$
 \end{theorem}
\begin{proof}
This proof is very analogous to that in Thereom B. It is obvious that $\Lambda^+,\Lambda^-$ and $\Lambda_7^6$ are complete prime due to Proposition \ref{allprimepin7d}. We divide the proof into three parts.
\begin{itemize}
\item[(1)] We first prove that $\Lambda_{27}^4$ and $\Lambda_8^3$ are complete prime, this is equivalent to that the two projections $\pi_8:\omega^1\cdot(\Lambda_1^4\oplus\Lambda_7^4\oplus\Lambda^-)
\rightarrow\Lambda_8^5$ and $\pi_{27}:\omega^1\cdot\Lambda_{48}^3\rightarrow\Lambda_{27}^4$ are surjective. Choose  a basis $\{\omega^k\cdot\Omega_0\}_{k=1}^8$ for $\Lambda_8^5.$ Observe that if $(\omega^k\cdot\Omega_0,\omega^{1ijuv})=1$ or $(\omega^k\cdot\Omega_0,\omega^1\cdot*\omega^{ijuv})=-1,$ then 
$$\pi_8(\omega^1(\omega^{ijuv}-*\omega^{ijuv}))=\frac{1}{7}\omega^k\cdot\Omega_0.$$
Hence, we have $\pi_8(\omega^1\cdot\Lambda^-)=\Lambda_8^5.$ We claim that\\
\textbf{Claim:} The projection $\pi_{27}:\omega^1\cdot\Lambda_{48}^3\rightarrow\Lambda_{27}^4$ is surjective.\\
After choosing two bases of $\Lambda_{48}^3$ and $\Lambda_{27}^4,$ the map $\pi_{27}$ can be written explicitly in a matrix which is full rank. It is not difficult but is too long to give here. 
\item[(2)] Take $\lambda=\omega^1,$ we next prove that the complexes
\begin{eqnarray}
&&0\longrightarrow\Lambda_7^2\stackrel{\pi_{8}\omega^1}{\longrightarrow}\Lambda_8^3\stackrel{\pi_{1}\omega^1}{\longrightarrow}\Lambda_1^4\rightarrow0,\label{spin72ec1} \\
&&0\longrightarrow\Lambda_8^3\stackrel{\pi_{1,7}\omega^1}{\longrightarrow}\Lambda_1^4\oplus\Lambda_7^4\longrightarrow0,\label{spin72ec2}   \\
&&0\longrightarrow\Lambda_1^4\oplus\Lambda_7^4\stackrel{\pi_{8}\omega^1}{\longrightarrow}\Lambda_8^5\longrightarrow0, \label{spin72ec3} \\
&&0\longrightarrow\Lambda_{21}^2\stackrel{\pi_{48}\omega^1}{\longrightarrow}\Lambda_{48}^3\stackrel{\pi_{27}\omega^1}{\longrightarrow}\Lambda_{27}^4\longrightarrow0 \label{spin72ec4}
\end{eqnarray}
are exact. Choose a basis $\{\alpha_i\}_{i=1}^7$ for $\Lambda_7^2$ and a basis $\{\beta_i\}_{i=1}^7$ for $\Lambda_7^4$ as
\begin{equation} \label{basislamb7}
\begin{split}
&\alpha_1=\omega^{12}+\omega^{34}+\omega^{56}+\omega^{78},\ \ \ 
\alpha_2=\omega^{13}-\omega^{24}+\omega^{57}-\omega^{68},\\
&\alpha_3=\omega^{14}+\omega^{23}-\omega^{58}-\omega^{67},\ \ \ 
\alpha_4=\omega^{15}-\omega^{26}-\omega^{37}+\omega^{48},\\
&\alpha_5=\omega^{16}+\omega^{25}+\omega^{38}+\omega^{47},\ \ \ 
\alpha_6=\omega^{17}-\omega^{28}+\omega^{35}-\omega^{46},\\
&\alpha_7=\omega^{18}+\omega^{27}-\omega^{36}-\omega^{45},\\
&\beta_1=\omega^{2467}+\omega^{2458}+\omega^{2368}-\omega^{2357}-\omega^{1358}-\omega^{1367}-\omega^{1457}+\omega^{1468},\\
&\beta_2=\omega^{3467}+\omega^{3458}+\omega^{2378}+\omega^{2356}+\omega^{1258}+\omega^{1267}+\omega^{1456}+\omega^{1478},\\
&\beta_3=\omega^{3457}-\omega^{3468}+\omega^{2478}+\omega^{2456}+\omega^{1257}-\omega^{1268}-\omega^{1356}-\omega^{1378},\\
&\beta_4=\omega^{2345}-\omega^{4567}-\omega^{3568}+\omega^{2578}-\omega^{1238}-\omega^{1247}+\omega^{1346}+\omega^{1678},\\
&\beta_5=\omega^{4568}-\omega^{3567}+\omega^{2678}+\omega^{2346}-\omega^{1237}+\omega^{1248}-\omega^{1345}-\omega^{1578},\\
&\beta_6=\omega^{4578}+\omega^{3678}+\omega^{2567}+\omega^{2347}+\omega^{1236}+\omega^{1245}+\omega^{1348}+\omega^{1568},\\
&\beta_7=\omega^{3578}-\omega^{4678}+\omega^{2568}+\omega^{2348}+\omega^{1235}-\omega^{1246}-\omega^{1347}-\omega^{1567}.
\end{split}
\end{equation}
Calculations show that the projections in (\ref{spin72ec1})-(\ref{spin72ec4}) satisfy 
\begin{eqnarray*}
&&\pi_{8}(\omega^1\alpha_i)=-\frac{3}{7}*(\Omega_0\omega^{i+1}),\ \ \ 
\pi_{1}(\omega^1*(\Omega_0\omega^1))=\frac{1}{2}\Omega_0,\\
&&\pi_{1,7}(\omega^1*(\Omega_0\omega^{i+1}))=\frac{1}{2}\beta_i,\ \ \ 
\pi_{8}(\omega^1\beta_{i+1})=\frac{4}{7}\Omega_0\omega^{i+1},
\end{eqnarray*}
for $i=1,2,\cdots,7.$ Therefore  (\ref{spin72ec1}), (\ref{spin72ec2}) and  (\ref{spin72ec3}) are exact. Because $\Lambda_8^3$ is prime, we know that $(\omega^1\Lambda_{21}^2)\cap\Lambda_8^3=\{0\},$ this implies $\pi_{48}\circ\omega^1$ is injective. Hence  (\ref{spin72ec4}) is exact.
\item[(3)] When $\mathrm{hol}(g)=\mathrm{Spin}(7),$ we check the prime spaces in Proposition \ref{allprimepin7d} one by one. Using the relations of $\Lambda_i^k$ in Proposition \ref{lakirela6}, one can see that $\Lambda_7^2$ is not complete prime. Besides, $\Lambda_{27}^4$ is the only complete prime proper subspace of $\Lambda^+.$ Hence there is no other complete prime spaces of type  \uppercase\expandafter{\romannumeral1}. Finally, except for $(\Lambda_7^2,\Lambda_{48}^3),(\Lambda_{21}^2,\Lambda_8^3),(\Lambda_8^3,\Lambda_{27}^4\oplus\Lambda^-)$ and $(\Lambda_1^4\oplus\Lambda_7^4,\Lambda_{48}^5),$ there is no more complete prime pairs of type  \uppercase\expandafter{\romannumeral2}. This completes the proof.
 \end{itemize}
 \end{proof}
 
 \section{Indexes of complexes}\label{lastsect77s}
In this section we compute the index of complexes in Theorem \ref{g2ellpcthe} and Theorem \ref{spin7ellpcthe}. Let $(M,g)$ be a compact Riemannian $n$-manifold with a compatible $G$-structure $Q$ and let $W$ be a subbundle of $\Lambda^k(T^*M)$ arising from a $G$-invariant subspace $\mathcal{W}$ of $\Lambda^k(\mathbb{R}^n).$ S. S. Chern showed in \cite{Chern1957} that the orthogonal projection $p:\Lambda^k(M)\rightarrow W$ commutes with Laplacian, i.e., $p\Delta=\Delta p,$ and one can define a refined cohomology group
$$\mathcal{H}(W):=\mathrm{Ker}(\Delta|_W)=\{\alpha\in\Gamma(W)|\mathrm{d}\alpha=\mathrm{d}^*\alpha=0\}.$$
Moreover, if $\mathcal{W}$ and $\mathcal{W}'$ are isomorphic representations of $G,$ then $\mathcal{H}(W)$ is isomorphic to $\mathcal{H}(W').$ When $\Lambda_i^k$ is an irreducible subbundle, we write $\mathcal{H}(\Lambda_i^k)$ as $\mathcal{H}_i^k.$ See \cite[Chap. 3]{Joyce2007} for details.

We write a Dolbeault complex 
$$0\rightarrow\Gamma(W_0)\rightarrow\Gamma(W_1)\rightarrow\Gamma(W_2)\rightarrow\cdots\rightarrow\Gamma(W_n)\rightarrow 0$$ 
as $(W_*,D),$ where $W_i$ is a subbundle (perhaps with rank zero) of $\Lambda^i(T^*M).$ If for every $i,$ the $i$-th cohomology group of $(W_*,D)$ is isomorphic to $\mathcal{H}(W_i),$ we say that $(W_*,D)$ is \emph{harmonic}.
\begin{theorem}\label{lastheind2}
All the Dolbeault complexes in Theorem \ref{g2ellpcthe} and Theorem \ref{spin7ellpcthe} are harmonic.
\end{theorem}
Let $A=(A_i,D)_{i=1}^r$ and $B=(B_i,D)_{i=r}^k$ be two complexes, where $A$ is harmonic. The top degree of $A$ and the nethermost degree of $B$ are the same. Then $B$ is harmonic if and only if $A\oplus B$ is harmonic. It follows from the short exact sequence $A\rightarrow A\oplus B\rightarrow B$ that 
\begin{equation}
H^i(A\oplus B)\cong\left\{\begin{aligned}
&&H^i(A)\quad\quad\quad\quad, i<r,\\
&&H^i(A)\oplus H^i(B), i=r,\\
&&H^i(B)\quad\quad\quad\quad, i>r.
\end{aligned}\right.
\end{equation}
Similarly, let $C:A^0\rightarrow\cdots\rightarrow A^{k-1}\rightarrow \Gamma(E_0)\rightarrow \Gamma(E_1)\rightarrow 0$ be a complex such that at least two of three isomorphisms $H^{k-1}(C)=H^{k-1}(M),H^{k}(C)\cong\mathcal{H}(E_0),H^{k+1}(C)\cong\mathcal{H}(E_1)$ are true. By considering the exact sequence $C\rightarrow A^*(M)\rightarrow\widehat{C},$ we see that $C$ is harmonic if and only if $\widehat{C}$ is harmonic. 
\subsection{Proof of Theorem \ref{lastheind2} for $G_2$-manifolds}
 Let $(M,g,\varphi)$ be a compact $G_2$-manifold. With canonical isomorphisms (\ref{canoisog27}) and $\Lambda_1^0\cong\Lambda_1^3\cong\Lambda_1^4\cong\Lambda_1^7,$ each operator $D:\Lambda_1^k\oplus\Lambda_7^l\rightarrow\Lambda_1^r\oplus\Lambda_7^s$ can be written as $\hat{D}:\Lambda^0\oplus\Lambda^1\rightarrow\Lambda^0\oplus\Lambda^1.$ 
  \begin{lemma}\label{g2lamb714}
Let $\lambda\in A^1.$ If $\pi_7\mathrm{d}\lambda=0$ or $\pi_{14}\mathrm{d}\lambda=0,$ then $\mathrm{d}\lambda=0.$ 
\end{lemma}
\begin{proof}
Since $\Lambda_7^2$ and $\Lambda_{14}^2$ are eigenspaces of $T_{\varphi_0}:\Lambda^2\rightarrow\Lambda^2,$ let $*(\varphi\wedge\mathrm{d}\lambda)=a\cdot\mathrm{d}\lambda,$ where $a=-1$ or $2.$ Then we have $\mathrm{d}\lambda=0$ from
$$(\mathrm{d}\lambda,\mathrm{d}\lambda)=\int_M\mathrm{d}\lambda\wedge *\mathrm{d}\lambda=\frac{1}{a}\int_M\mathrm{d}(\lambda\wedge\varphi\wedge\mathrm{d}\lambda)=0.$$
 \end{proof}
 Define a differential operator $\Theta:A^1\rightarrow A^1$ on $G_2$-manifold by $\Theta(\lambda)=*(*\varphi\wedge\mathrm{d}\lambda).$ By Lemma \ref{g2lamb714} and the fact that $A_{14}^2$ is the kernel of the multiplication $A^2\rightarrow A^6$ by $*\varphi,$ the kernel of $\Theta$ is the space of all closed $1$-forms. For every $p\in M,$ choose a coordinate system $(x_1,\cdots,x_7)$ on $U\subset M$ near $p$ such that $(\frac{\partial}{\partial x_i},\frac{\partial}{\partial x_j})_p=\delta_{ij}.$ By taking $\lambda|_U=x_i\mathrm{d}x_j,$ it is not hard to check that the following equalities hold.
 \begin{proposition}\label{indpropg2pi1}
Let $\lambda\in A^1,\alpha=*(*\varphi\wedge\lambda)\in A_7^2$ and $\beta=*(\varphi\wedge\lambda)\in A_7^3.$ Then 
 \begin{eqnarray*}
&&  \pi_1 \mathrm{d}\alpha=-\frac{3}{7}\mathrm{d}^*\lambda\cdot\varphi,\ \  \pi_1\mathrm{d}\beta=-\frac{4}{7}\mathrm{d}^*\lambda\cdot *\varphi,\\
 && \pi_7\mathrm{d}\lambda=\frac{1}{3}*(*\varphi\wedge\Theta(\lambda)),\ \  \pi_7\mathrm{d}\alpha=\frac{1}{2}*(\varphi\wedge\Theta(\lambda)),\\
 && \pi_7\mathrm{d}\beta=\frac{1}{2}\varphi\wedge\Theta(\lambda),\ \  \pi_7\mathrm{d}^*\alpha=\Theta(\lambda),\\
 && \pi_7\mathrm{d}^*\beta=\frac{2}{3}*(*\varphi\wedge\Theta(\lambda)),\ \  \pi_7\mathrm{d}*\lambda=\frac{1}{3}*\varphi\wedge\Theta(\lambda).
 \end{eqnarray*}
 \end{proposition}
From Proposition \ref{indpropg2pi1} we immediately deduce 
\begin{eqnarray}
\pi_1\mathrm{d}\alpha=0\ \Leftrightarrow\!\!&\pi_1\mathrm{d}\beta=0&\!\!\Leftrightarrow\ \mathrm{d}^*\lambda=0,\\
\pi_7\mathrm{d}\alpha=0\ \Leftrightarrow\ \mathrm{d}^*\alpha=0\ \Leftrightarrow\!\!&\pi_7\mathrm{d}\beta=0&\!\! \Leftrightarrow\ \mathrm{d}^*\beta=0\ \Leftrightarrow\ \mathrm{d}\lambda=0.
 \end{eqnarray}
 By Proposition \ref{indpropg2pi1} , we see that (\ref{g2ellipticom0123}),(\ref{g2ellipticom234}),(\ref{g2ellipticom34}) and $A_1^4\rightarrow A_7^5\rightarrow A^6\rightarrow A^7$ (denoted by $K_2$) are harmonic. It remains to consider the complex (\ref{g2ellipticom14}) (denoted by $K_1$) which is a quotient complex of $A^*(M)$ by $K_2.$ Assume $\alpha\in H^3(K_1)$ with $\mathrm{d}\alpha\in A_1^4.$ Then $\mathrm{d}\alpha=0$ because $*\varphi$ is a prime and not exact form. This means $H^3(K_1)=H^3(M).$ Assume $\beta\in A_{14}^5$ with $\mathrm{d}^*\beta\in A_1^4.$ Similarly, we have $\mathrm{d}^*\beta=0.$ It follows from $\varphi\wedge*\beta+\beta=0$ that $\mathrm{d}\beta=-\varphi\wedge*\mathrm{d}^*\beta=0.$ This means $H^5(K_1)=\mathcal{H}_{14}^5.$ To compute $H^4(K_1),$ cosider
the short exact sequence $0\rightarrow K\rightarrow A^*(M)\rightarrow K_1\rightarrow 0$ which induces a long exact sequence
 $$\cdots\rightarrow H^3(M)\stackrel{i}{\rightarrow} H^3(K_1)\rightarrow\mathcal{H}_1^4\rightarrow H^4(M)\rightarrow H^4(K_1)\rightarrow\mathcal{H}_7^5\stackrel{j}{\rightarrow}H^5(M)\rightarrow\mathcal{H}_{14}^5\rightarrow\cdots.$$
Since $i$ is surjective and $j$ is injective, we have $H^4(K_1)\cong H^4(M)/\mathcal{H}_1^4\cong\mathcal{H}_7^4\oplus\mathcal{H}_{27}^4.$
\subsection{Proof of Theorem \ref{lastheind2} for $\mathrm{Spin}(7)$-manifolds}
Observe that $C_3$ satisfies $\widehat{C_1\oplus C_3}=C_4\oplus\widehat{C}_-,$ all complexes in Theorem \ref{spin7ellpcthe} are generated by $C_1,C_2,C_4,C_5,C_6,C_\pm.$ By Example \ref{examplesign}, $C_\pm$ are harmonic. We next explain that $C_1,C_2,C_4,C_5,C_6$ are Dirac operators or the rearrangements of Dirac operators.

Let $e_1,\cdots,e_8$ be an oriented orthonormal basis of $\mathbb{R}^8.$ Recall from \cite{Harvey1993} or \cite{Jianwei2002} that the spin representations of the real Clifford algebra $Cl_8$ can be constructed in the following way. Define $A=A_8(e_1e_3e_5e_7-1),$ where
 $$A_8=\mathrm{Re}(e_1+\sqrt{-1}e_2)(e_3+\sqrt{-1}e_4)(e_5+\sqrt{-1}e_6)(e_7+\sqrt{-1}e_8).$$
Then $V_8^+=Cl_8^{\mathrm{even}}\cdot A$ and $V_8^-=Cl_8^{\mathrm{odd}}\cdot A$ are two spin representations of $Cl_8.$ Moreover, $V_8^+$ has basis $\{\sigma_i=e_1e_iA\}_{i=1}^8$ and $V_8^-$ has basis $\{\sigma_i'=-e_iA\}_{i=1}^8.$ With the canonical isomorphism $\varphi:\Lambda^*(\mathbb{R}^8)\rightarrow Cl_8,$ we have 
\begin{equation}\label{identiv8pm}
\begin{split}
&\sigma_1=-1+\varphi(\Omega_0)+e_1e_2\cdots e_8,\\
&\sigma_{i+1}=\varphi(\alpha_i)+\varphi(\beta_i)-\varphi(*\alpha_i),\\
&\sigma_j'=-\varphi(\omega^j)-\varphi(*(\Omega_0\cdot\omega^j))+\varphi(\Omega_0\cdot\omega^j)+\varphi(*\omega^j),
\end{split}
\end{equation}
for $i=1,\cdots,7$ and $j=1,\cdots,8,$ where $\alpha_i$ and $\beta_i$ are basis of $\Lambda_7^2$ and $\Lambda_7^4$ given in (\ref{basislamb7}), and $\Omega_0$ is a $4$-form given in (\ref{Omega0def}). Then (\ref{identiv8pm}) induces natural isomorphisms $V_8^+\cong\Lambda_1^i\oplus\Lambda_7^j,V_8^-\cong\Lambda_8^k$ for $i\in\{0,4,8\},j\in\{2,4,6\}$ and $k\in\{1,3,5,7\}.$

As we all know, every $8$-manifold $M$ with a torsion-free $\mathrm{Spin}(7)$-structure $(\Omega,g)$ has a spin structure. The spin bundle $S=S^+\oplus S^-$ of $M$ corresponds to the representations $V_8^+,V_8^-$ of $\mathrm{Spin}(8).$ Then $S^+\cong \Lambda^0\oplus\Lambda_7^2\cong\Lambda_1^4\oplus\Lambda_7^2\cong\Lambda_1^4\oplus\Lambda_7^4$ and $S^-\cong\Lambda_8^k,k=1,3,5,7.$ D. Joyce showed in \cite[\S 6]{Joyce1996} that the Dirac operator $D_+$ can be written as
\begin{eqnarray*}
D_+(\xi_1,\xi_7)=8\mathrm{d}\xi_1+7\pi_8\mathrm{d}\xi_7:\Lambda_1^4\oplus\Lambda_7^4\rightarrow\Lambda_8^5.
\end{eqnarray*}
Also, $C_1,C_2,C_4$ and $C_6$ are all equivalent to the Dirac operator or its rearrangements. Since the Ricci curvature of $\mathrm{Spin}(7)$-manifold is zero, the Weitzenbock formula of Lichnerowicz \cite{Lichnerowicz1963} shows that spinors in $\mathrm{Ker}D_+$ are constant. Thus, $C_1,C_2,C_4,C_5,C_6$ are harmonic.

Finally, we need to prove the cases for $C_3$ and $C_1\oplus C_3.$
\begin{lemma}\label{lastlemhh35}
\begin{itemize}
\item[(1)] Suppose $\alpha\in A_{21}^2$ with $\mathrm{d}\alpha\in A_8^3.$ Then $\mathrm{d}\alpha=\mathrm{d}^*\alpha=0.$
\item[(2)] Suppose $\beta\in A_{27}^4$ with $\mathrm{d}^*\beta\in A_{8}^3.$ Then $\mathrm{d}\beta=\mathrm{d}^*\beta=0.$
\end{itemize}
\end{lemma}
\begin{proof}
\begin{itemize}
\item[(1)] Set $\mathrm{d}\alpha=*(\Omega\wedge\lambda)$ for some $\lambda\in A^1.$ Then $\mathrm{d}*(\Omega\wedge\lambda)=0.$ Since $A_8^3\rightarrow A_1^4\oplus A_7^4$                                                                                                                  
 is harmonic, then $\mathrm{d}^*(\Omega\wedge\lambda)=0$ leads to $\mathrm{d}\lambda=0.$ We have $\mathrm{d}\alpha=0$ from
 $$(\mathrm{d}\alpha,\mathrm{d}\alpha)=-\int_M\mathrm{d}\alpha\wedge\Omega\wedge\lambda=-\int_M\mathrm{d}(\alpha\wedge\Omega\wedge\lambda)=0.$$
 It follows from $*(\Omega\wedge\alpha)=3\alpha$ that $\mathrm{d}^*\alpha=-\frac{1}{3}*(\Omega\wedge\mathrm{d}\alpha)=0.$             
   \item[(2)] Proposition \ref{lakirela6} shows that $\Lambda_{27}^4\cdot\Lambda^1=\Lambda_{48}^5,$ therefore, $\mathrm{d}\beta=\mathrm{d}^*\beta=0$     \end{itemize}\end{proof}
  A consequence of Lemma \ref{lastlemhh35} is that $H^2(C_1\oplus C_3)=H^2(M)$ and $H^4(C_1\oplus C_3)=H_{27}^4.$ Using the fact that $\widehat{C_1\oplus C_3}=C_4\oplus\widehat{C}_-,$ we find that $C_1\oplus C_3$ and $C_3$ are harmonic.

\textbf{Acknowledgements:} The author would like to thank D.D. Joyce for explanations about the dirac operator on $\mathrm{Spin}(7)$-manifold. The author is grateful to Prof. J. Zhou for many useful suggestions. The author is also indebted to H. J. Fan for the guidance over the past years.

\end{document}